\documentclass[10pt,a4paper]{article}

\usepackage{graphicx}
\usepackage{algorithm,algorithmic}
\usepackage{multirow}
\usepackage{setspace}
\usepackage{amsmath,amssymb,amsthm,amsfonts,subfigure,colordvi,xcolor}
\usepackage{cite}
\usepackage{soul}
\usepackage[utf8]{inputenc}

\usepackage[normalem]{ulem}

\newtheorem{lemma}{Lemma}[section]
\newtheorem{theorem}{Theorem}[section]
\newtheorem{definition}{Definition}[section]
\def\f{\mathbf{f}}
\def\s{\mathbf{s}}
\title{2DNMR data inversion using locally adapted multi-penalty regularization}
\author{V. Bortolotti, G. Landi, F. Zama}
\begin{document}

\maketitle
\begin{abstract}
A crucial issue in two-dimensional Nuclear Magnetic Resonance (NMR) is the speed and accuracy of the data inversion.
This paper proposes a multi-penalty method with locally adapted regularization parameters
for fast and accurate inversion of 2DNMR data.

The method solves an unconstrained optimization problem whose objective contains a data-fitting term, a single  $L1$ penalty parameter and a multiple parameter $L2$ penalty. 
We propose an adaptation of the Fast Iterative Shrinkage and Thresholding (FISTA)  method to solve the multi-penalty minimization problem, and an automatic procedure to compute all the penalty parameters.
This procedure generalizes the Uniform Penalty principle introduced in [Bortolotti et al., \emph{Inverse Problems}, 33(1), 2016].  

The proposed  approach allows us to obtain  accurate relaxation time distributions while keeping short the 
computation time. Results of numerical
experiments on synthetic and real data prove that
the proposed method is efficient and effective in reconstructing the peaks and the flat regions
that usually characterize NMR relaxation time distributions.
\end{abstract}

%\begin{keyword}
%Inverse Problems 	\sep  multi-parameter regularization \sep  multi-penalty regularization \sep  NMR relaxometry \sep 2DUPEN \sep   FISTA
%\end{keyword}
%
%\end{frontmatter}
\section{Introduction}
The inversion of Nuclear Magnetic Resonance (NMR) relaxation data of $^1$H nuclei is a crucial technique to analyze the structure of porous media, ranging from cement to biological systems.
In 2DNMR, joint measurements of the spin relaxation with respect to the longitudinal and transverse relaxation parameters $T_1$ and $T_2$
allow us to build two-dimensional relaxation time distributions. Peaks usually characterize such distributions over flat regions; the position and
volume of the peaks are used to obtain information such as petrophysical properties, molecular diffusion,  \cite{MITCH2014}.
The measured NMR data are related to the
relaxation time distribution according to a Fredholm integral equation
of the first kind with separable exponential kernel. Due to the large dimension of the data and the inherent ill-posedness of the inverse problem, a significant issue in 2DNMR inversion is to ensure both computational efficiency and accuracy.
This aspect is particularly relevant in multidimensional logging where 3DNMR inversion algorithms are usually based on methods for 2DNMR inversion. Therefore, the development and application of 3DNMR techniques is seriously restricted by the efficiency and accuracy of the 2D inversion \cite{ZHANG2019174}.

In a discrete setting, the 2D Fredholm integral equation can be modeled as a linear inverse problem
\begin{equation}
\mathbf{K} \mathbf{f} + \mathbf{e}= \mathbf{s}
\label{mod1}
\end{equation}
where $\mathbf{K}= \mathbf{K}_2 \otimes \mathbf{K}_1$ is the Kronecker product of the discretized
decaying exponential kernels $\mathbf{K}_1\in \mathbb{R}^{M_1 \times N_1}$ and $\mathbf{K}_2\in \mathbb{R}^{M_2 \times N_2}$ .
The vector  $\mathbf{s} \in \mathbb{R}^{M}$, $M=M_1 \cdot M_2$, represents
the measured noisy signal, $\mathbf{f} \in   \mathbb{R}^{N}$, $N=N_1 \cdot N_2$, is the vector
reordering of the 2D distribution to be computed and $ \mathbf{e} \in  \mathbb{R}^{M}$ represents the
additive Gaussian noise.
The severe ill-conditioning of $\mathbf{K} $ is well-known, and it causes the least-squares solution of \eqref{mod1} to be
extremely sensitive to the noise; for this reason, regularization is usually applied.
The most common numerical strategies are based on $L2$ regularization
and often use constraints, such as non-negativity constraints, in order to prevent unwanted distortions
in the computed distribution. This approach requires
solving the nonnegatively constrained Tikhonov-like problem:
\begin{equation}
  \min_{\f\geq 0} \left \{  \| \mathbf{K} \f - \s \|^2 +\lambda \| \f \|^2 \right \}
  \label{eq:L2}
\end{equation}
where $\lambda>0$ is the regularization parameter  %the nonnegative constraint imposes some physical bound on the unknown distribution
and $\| \cdot \|$ denotes the Euclidean norm.
In this context, the approach of Venkataramanan et al. \cite{Venta_Song_2002},
uses data compression to reduce the size of problem \eqref{eq:L2} and the Butler–Reeds–Dawson method \cite{BRD_1981} to solve the smaller-size optimization problem. Chouzenoux et al. \cite{Chouzenoux} apply the interior point method for the solution of \eqref{eq:L2}.
The main drawback of single parameter $L2$ regularization is its tendency to either over-smooth the solution, making it difficult
to detect low-intensity peaks, or to under-smooth the solution creating non-physical sharp peaks.
Substantial improvements are obtained by the application of multiple parameters Tikhonov
regularization, as in the 2DUPEN algorithm \cite{Bortolotti_2016, Bortolotti2019} which solves the minimization problem
\begin{equation}\label{eq:multiTikh}
  \min_{\f \geq 0} \left \{ \| \mathbf{K} \f - \s \|^2 + \sum_{i=1}^N \lambda_i(\mathbf{L} \mathbf{f})^2_i\right \}
\end{equation}
where $\mathbf{L}\in \mathbb{R}^{N}$ is the discrete Laplacian operator. The multiple regularization parameters $\lambda_i$s are locally adapted, i.e., at each iteration,
approximated values for the $\lambda_i$s are computed by imposing the
Uniform Penalty (UPEN) principle \cite{Bortolotti_2016}
and a constrained subproblem is solved by the
Newton Projection method \cite{ber82Siam}.
Although 2DUPEN can obtain very accurate distributions, as reported in the literature  \cite{Bortolotti_2016,Bortolotti_2017,micro_meso2018,Bortolotti2019}, its computational cost may be high
since it requires the solution of several nonnegatively constrained least-squares problems.

In the NMR literature, $L1$ regularization has been recently considered in order to better reproduce the characteristic sparsity of the relaxation distribution. In \cite{ZHOU_2017}, the $L1$ regularization problem
\begin{equation}\label{eq:L1}
  \min_{\f} \left \{  \| \mathbf{K} \f - \s \|^2 +\alpha \| \f \|_1 \right \}
\end{equation}
is considered and the Fast Iterative Shrinkage-Thresholding Algorithm (FISTA) \cite{Beck_Teb_2013} is used for its solution.
An update searching method is proposed to iteratively determine the
regularization parameter as
$\alpha = \sqrt{N}\sigma/\|\f\|_1$
where $\sigma$ is the standard deviation of the noise. 
We remark that FISTA is known to be one of the most effective and efficient methods for solving $L1$-based image denoising and deblurring problems. Recently, FISTA has also been applied to non-convex regularization \cite{AMC19, Lazzaro_2019}.

The $L1$ regularization has also been used in NMR to decrease the data acquisition time  \cite{WuDag2014}. In \cite{Teal_2015}, an algorithm related to FISTA is applied to
NMR relaxation estimation and comparisons
with the methods of Venkataramanan et al. \cite{Venta_Song_2002}, and Chouzenoux et al. \cite{Chouzenoux}, are carried out showing the efficiency of the FISTA-like method.
However, despite its computational efficiency and its capability of revealing isolated narrow peaks,  $L1$ regularization tends to divide a wide peak or tail into separate undesired peaks.

Recently, the elastic net method \cite{ZouHastle2005} with a non-negative constraint
\begin{equation}\label{eq:elasticnet}
  \min_{\f \geq 0} \left \{  \| \mathbf{K} \f - \s \|^2 +\lambda \| \f \|^2 +\alpha \| \f \|_1 \right \}
\end{equation}
has been used in \cite{Berman_et_al_2013} to obtain
$T2$ distributions. The problem is formulated as
a linearly constrained convex optimization
problem and the primal-dual interior method for convex objectives has been applied to one dimensional \cite{Berman_et_al_2013} and two dimensional \cite{Campisi-Pinto2018} NMR relaxation problems.
However, the performance of the method depends on the
two regularization parameters which needs an accurate tuning; a parameter selection analysis is performed in \cite{Campisi-Pinto2018} for a specific set of 2DNMR data.

Our previous review of the current literature shows that for each inversion method, we have to take into account both the efficiency and the accuracy. The 2DUPEN method has a great inversion accuracy due to the employment of multi-penalty regularization with locally adapted parameters, but the nonnegative constraints are responsible for its poor computational efficiency. The $L1$ regularization with FISTA algorithm is computationally very efficient, but its accuracy can be low in the presence of non-isolated peaks. Multi-penalty regularization \eqref{eq:elasticnet} is able to simultaneously promote
distinct features of the sought-for distribution, since it yields a good
trade-off among data fitting error, sparsity and smoothness
of the solution. However, its applicability is greatly limited by the fact that multiple parameters tuning is a challenging task depending
on SNR, sparsity and smoothness.
To overcome the aforementioned drawbacks   ensuring both efficiency and accuracy, in this paper, we propose a multi-penalty approach involving $L1$ and $L2$ penalties with locally adapted regularization parameters. The proposed method can be mathematically formulated as the unconstrained minimization problem
\begin{equation}
  \min_{\mathbf{f}} \left \{ \|  \mathbf{K} \mathbf{f} - \mathbf{s} \|^2 +\sum_{i=1}^N \lambda_i(\mathbf{L} \mathbf{f})^2_i+ \alpha \|\mathbf{f}\|_1 \right \} .
  \label{eq:uno}
\end{equation}
This approach allows us to accurately reconstruct distributions with isolated and non-isolated peaks as well as flat areas, in short computation time. On the one hand, $L1$ regularization prevents from over-smoothing while, on the other hand, local $L2$ regularization prevents from under-smoothing merging peaks or peak tails. Since $L1$ regularization enforces sparse distributions, the nonnegative constraints are not included in problem \eqref{eq:uno} and FISTA can be used for its efficient and effective solution.
The UPEN principle is extended to the multi-penalty problem \eqref{eq:uno} obtaining a very efficient computation strategy for all the regularization parameters. Therefore, tedious multiple parameters tuning procedure is not necessary.

The contribution of this paper is two-fold. Firstly, it introduces a locally adapted multi-penalty model for NMR data inversion, and it proposes an efficient strategy for the automatic computation of the multiple parameters. Secondly, we prove that the solution of \eqref{eq:uno} is a regularized solution of problem \eqref{mod1}.
The extension of the regularization properties of the UPEN  principle to a
multiple regularization context makes it possible to apply it to more general, non-differentiable and possibly non-convex penalties.

The proposed algorithm has been tested on both synthetic and real NMR relaxometry problems, and has been compared to multiple parameters $L2$ regularization \eqref{eq:multiTikh} (2DUPEN) and to $L1$ regularization \eqref{eq:L1}. The numerical results show the efficiency and effectiveness of the method.

The remainder of the paper is organized as follows: section \ref{principle} analyzes the regularization properties of the proposed method. Section \ref{Upen-L1} reports the details of the numerical algorithm. Finally, in section \ref{num-res}, some results are shown and discussed both on synthetic and real NMR data.
A crucial issue in two-dimensional Nuclear Magnetic Resonance (NMR) is the speed and accuracy of the data inversion.
This paper proposes a multi-penalty method with locally adapted regularization parameters
for fast and accurate inversion of 2DNMR data.

The method solves an unconstrained optimization problem whose objective contains a data-fitting term, a single  $L1$ penalty parameter and a multiple parameter $L2$ penalty. 
We propose an adaptation of the Fast Iterative Shrinkage and Thresholding (FISTA)  method to solve the multi-penalty minimization problem, and an automatic procedure to compute all the penalty parameters.
This procedure generalizes the Uniform Penalty principle introduced in [Bortolotti et al., \emph{Inverse Problems}, 33(1), 2016].  

The proposed  approach allows us to obtain  accurate relaxation time distributions while keeping short the 
computation time. Results of numerical
experiments on synthetic and real data prove that
the proposed method is efficient and effective in reconstructing the peaks and the flat regions
that usually characterize NMR relaxation time distributions.%-----------------------------------------------------
\section{The Uniform Penalty Principle \label{principle}}
In order to generalize to multi-penalty regularization the UPEN principle introduced in \cite{Bortolotti_2016} for problem \eqref{eq:multiTikh}, let us write problem \eqref{eq:uno} as
\begin{equation}\label{eq:multipen}
  \min_{\mathbf{f}} \left \{ \|  \mathbf{K} \mathbf{f} - \mathbf{s} \|^2 +\sum_{i=1}^{N+1} \eta_i\phi_i(\f) \right \}
\end{equation}
where
\begin{equation}\label{}
  \phi_i(\f)=\left\{
               \begin{array}{ll}
                 (\mathbf{L} \mathbf{f})^2_i, & i=1,\ldots,N, \\
                 \|\f\|_1, & i=N+1,
               \end{array}
             \right.
\quad \text{and} \quad
  \eta_i=\left\{
               \begin{array}{ll}
                 \lambda_i, & i=1,\ldots,N, \\
                 \alpha, & i=N+1.
               \end{array}
             \right.
\end{equation}
%
%The uniform penalty principle, introduced in \cite{Bortolotti_2016} for multiple parameters Tikhonov regularization \eqref{eq:multiTikh},  is stated here for the sake of clarity
The generalization of the UPEN principle can be stated as follows.
\begin{definition}[Generalized Uniform Penalty Principle] \label{eq:UPP2_def} Choose the regularization parameters $\eta_i$
of multi-penalty regularization \eqref{eq:multipen} such that, at a solution $\mathbf{f}$,
the terms $\eta_i\phi_i(\f)$ are constant for all $i$ with
$\phi_i(\f)\neq0$, i.e:
\begin{equation}\label{eq:UPP2}
    \eta_i\phi_i(\f) = c, \quad \forall \; i=1,\ldots,N+1 \quad \text{s.t.} \quad \phi_i(\f)\neq 0
\end{equation}
where $c$ is a positive constant.
\end{definition}
Let us assume that a suitable bound $\varepsilon$ on the fidelity term of the exact solution $\mathbf{f}^*$ is given; i.e:
\begin{equation}\label{miller1}
    \|\mathbf{K} \mathbf{f}^* - \mathbf{s} \|^2 \leq \varepsilon^2
\end{equation}
where $\mathbf{f}^*$ is the solution of the noise-free least-squares problem
\begin{equation}\label{eq:noisefree}
\min_\f \left \{ \|\mathbf{K} \mathbf{f} - \hat{\mathbf{s}} \|^2, \; \mathbf{s} = \hat{\mathbf{s}}+\mathbf{e}\right\} .
\end{equation}
Following the Miller's criterium \cite{miller1970}, the constant $c$ is selected to balance the fidelity and regularization terms in \eqref{eq:multipen}; i.e:
\begin{equation}\label{eq:C}
    c=\frac{\varepsilon^2}{N_0}
\end{equation}
where $N_0$ is the number of non null terms $\phi_i(\f)$:
\begin{equation}\label{}
  N_0 = \# \{i \; | \; \phi_i(\f) \neq 0, \; i=1,\ldots,N+1\} .
\end{equation}
Obviously, with this choice for $c$, Lemma 3.1 of \cite{Bortolotti_2016}  still applies. The lemma is restated here for the sake of clarity.
\begin{lemma}
If $\mathbf{f}$ satisfies $\|\mathbf{K} \mathbf{f} - \mathbf{s} \|^2\leq\varepsilon^2$ and
the parameters $\eta_i$, $i=1,\ldots,N+1$,  are chosen according to the generalized uniform penalty principle with
\begin{equation}\label{eq:C_bis}
    c=\frac{\varepsilon^2}{N_0}
\end{equation}
where $N_0$ is the number of non null terms $\phi_i(\f)$, then
\begin{equation}\label{eq:2epsilon}
    \|\mathbf{K} \mathbf{f} - \mathbf{s} \|^2  + \sum_{i=1}^{N_0} \eta_i\phi_i(\f) \leq 2\varepsilon^2.
\end{equation}
Conversely, if $\mathbf{f}$ satisfies \eqref{eq:2epsilon} and the generalized UPEN principle with
\eqref{eq:C_bis}, then it  also satisfies $\|\mathbf{K} \mathbf{f} - \mathbf{s} \|^2\leq\varepsilon^2$.
\end{lemma}
\begin{proof}
Let $\mathbf{f}$ be such that $\|\mathbf{K} \mathbf{f} - \mathbf{s} \|^2\leq\varepsilon^2$, then, if
\eqref{eq:UPP2} holds with $c$ selected as in \eqref{eq:C_bis},  we have
%\begin{linenomath}
\begin{equation}
    \|\mathbf{K} \mathbf{f} - \mathbf{s} \|^2  + \sum_{i=1}^{N_0} \eta_i\phi_i(\f)
    \leq
    \varepsilon^2  + \sum_{i=1}^{N_0} \frac{\varepsilon^2}{N_0} = 2\varepsilon^2.
\end{equation}
%\end{linenomath}
Conversely, if \eqref{eq:2epsilon} and  \eqref{eq:C_bis} hold, then
\begin{equation}
    2\varepsilon^2 \geq \|\mathbf{K} \mathbf{f} - \mathbf{s} \|^2  + \sum_{i=1}^{N_0} \eta_i\phi_i(\f)
    = \|\mathbf{K} \mathbf{f} - \mathbf{s} \|^2   + \sum_{i=1}^{N_0} \frac{\varepsilon^2}{N_0} = \|\mathbf{K} \mathbf{f} - \mathbf{s} \|^2+\varepsilon^2.
    \end{equation}
\end{proof}
From \eqref{eq:UPP2} and \eqref{eq:C} we obtain the following expression for the $\eta_i$'s:
\begin{equation}\label{}
  \eta_i = \frac{\varepsilon^2}{N_0\phi_i(\f)} \quad \text{for all} \quad i=1,\ldots,N+1 \quad \text{such that} \quad \phi_i(\f)\neq 0
\end{equation}
which can be written in terms of the parameters $\lambda_i$ and $\alpha$ as
\begin{equation}\label{eq:parametri}
  \lambda_i = \frac{\varepsilon^2}{N_0 (\mathbf{L}\f)_i^2} \;\text{ if }\; (\mathbf{L}\f^*)_i\neq 0 \quad \text{and} \quad
  \alpha = \frac{\varepsilon^2}{N_0\| \f\|_1}.
\end{equation}
If the regularization parameters are computed as in \eqref{eq:parametri}, the following lemma shows that the solution of \eqref{eq:uno} is a regularized solution of \eqref{mod1}.
%
% by extending  the result  proven in \cite{Bortolotti_2016} Lemma 3.2.
%
\begin{lemma}
Let  $\mathbf{f}^*$ be the solution of the noise-free least-squares problem \eqref{eq:noisefree} and let $\mathbf{f}_{\varepsilon}$ denote the solution to problem \eqref{eq:uno} where the regularization parameters are chosen according to the generalized uniform penalty principle as follows:
\begin{equation}\label{eq:par}
 \lambda_i = \left\{
               \begin{array}{ll}
                 \displaystyle{\frac{\varepsilon^2}{N_0 (\mathbf{L}\f^*)_i^2}}, & \hbox{if $(\mathbf{L}\mathbf{f}^*)_i\neq0$;} \\
                 \gamma \varepsilon^2, & \hbox{otherwise;}
               \end{array}
             \right.
\quad \text{and} \quad
\alpha = \frac{\varepsilon^2}{ \| \f^*\|_1}
 \end{equation}
where $\gamma$ is a positive constant and $N_0$ is the number of non null terms $(\mathbf{L}\f^*)_i$.
Then
\begin{equation*}
  \lim_{\varepsilon \rightarrow 0}\mathbf{f}_{\varepsilon}= \mathbf{f}^*
  %\ \
  % \hbox{s.t.} \ \  \mathbf{f}^*=\arg\min_{\mathbf{f}} \| \mathbf{K} \mathbf{f} -\hat{\mathbf{s}} \|^2
\end{equation*}
and hence $\mathbf{f}_{\varepsilon}$ is a regularized solution of \eqref{mod1}.
\end{lemma}
\begin{proof}
Let us define the diagonal matrix $\mathbf{\Lambda}$  whose diagonal elements are the parameters $\lambda_i$.
%\begin{equation}\label{}
%    D_{i,i}=\left\{
%             \begin{array}{ll}
%               \lambda_i, & \hbox{if $(\mathbf{L}\mathbf{f}^*)_i\neq0$;} \\
%               \gamma \varepsilon^2 , & \hbox{otherwise.}
%             \end{array}
%           \right.
%\end{equation}
The first-order optimality conditions of \eqref{eq:uno} are
\begin{equation}
\{\mathbf{0} \} \in  2\mathbf{K} ^T \left (\mathbf{K} \f - \mathbf{s} \right ) +
2 \mathbf{L}^T\mathbf{\Lambda}\mathbf{L}\f + \alpha \mathbf{g}
\label{eq:neq_1}
\end{equation}
where $\mathbf{g}$ is the subgradient of  $\| \mathbf{f} \|_1$, i.e.:
\begin{equation*}
   g_i = \left \{ \begin{array}{ll}
                   +1, & \text{if } f_i >0 \\
                   -1, & \text{if } f_i <0 \\
                   \pm 1, & \text{if } f_i =0
                  \end{array}, \right.  \quad i=1,\ldots,N.
\end{equation*}
In the limit for $\varepsilon\rightarrow 0$, from \eqref{eq:parametri}, equation \eqref{eq:neq_1} becomes
\begin{equation}
\{\mathbf{0}\} \in  2\mathbf{K} ^T \left (\mathbf{K} \f - \mathbf{s} \right )
\end{equation}
which are the first-order optimality conditions for \eqref{eq:noisefree}.
%From \eqref{eq:upen} we obtain that $\lambda_i \propto \varepsilon^2$ and $\alpha \propto\varepsilon^2$;
%hence, as $\varepsilon$ goes to zero,
%setting $\tilde \f \equiv \lim_{\varepsilon \rightarrow 0}$
%we obtain from \eqref{eq:neq_1}:
% $$ \{0 \} \in  \mathbf{K} ^t \left (\mathbf{K} \tilde \f  - {\hat \s} \right )$$
% hence $$\tilde \f  = \arg\min_{\mathbf{f}} \| \mathbf{K} \mathbf{f} -\hat{\mathbf{s}} \|^2$$
\end{proof}
\section{The proposed method \label{Upen-L1}}
The computation of the parameters $\lambda_i$, $i=1,\ldots,N$, and $\alpha$ as in \eqref{eq:par} uses the quantities $\mathbf{f}^*$ and $\varepsilon$ which are unknown. For this reason, we propose a splitting iterative procedure where they are respectively approximated by the $k-$th iterate $\mathbf{f}^{(k)}$ and the corresponding residual norm $\|\mathbf{K} \mathbf{f}^{(k)} - \mathbf{s} \|$.
The proposed iterative procedure is outlined in Algorithm~\ref{alg:1} where $\rho$ is a small threshold parameter introduced in order to prevent divisions by zero and $\tau$ is a tolerance of the stopping criterium.
\begin{algorithm}[h]
\caption{\label{alg:1}}
\vskip 1mm
\begin{spacing}{1.15}
%{\small
\begin{algorithmic}[1]
\STATE Compute a starting guess $\f^{(0)}$;
\STATE Choose $\rho,\tau\in(0,1)$; set $k=0$;
\REPEAT
\smallskip
  \STATE Set $\epsilon^{(k)}=\|\mathbf{K} \mathbf{f}^{(k)} - \mathbf{s} \|^2$
\smallskip
  \STATE Set $\lambda_i^{(k)}=\displaystyle{\frac{\epsilon^{(k)}}{(N+1) (\mathbf{L} \f^{(k)})^2_i+ \rho}}$, $i=1, \ldots N$
\smallskip
  \STATE Set $\alpha^{(k)} = \displaystyle{\frac{\epsilon^{(k)}}{(N+1)  \| \f ^{(k)}\|_1}}$
\smallskip
  \STATE Compute
  \begin{equation*}
    \f ^{(k+1)}= \arg \min_{\f} \; \left \{ \|\mathbf{K} \mathbf{f} - \mathbf{s} \|^2  + \sum_{i=1}^N \lambda^{(k)}_i(\mathbf{L}\mathbf{f})_i^2 +\alpha^{(k)} \|\mathbf{f}\|_1  \right \}
  \end{equation*}
  \STATE Set $k=k+1$
\UNTIL{$\|\f ^{(k+1)} - \f ^{(k)} \| \leq \tau \|  \f ^{(k)} \|$}
\end{algorithmic}
%} % end small
\end{spacing}
\end{algorithm}

The computation of each new approximate solution $\f ^{(k+1)}$ at step 7 of Algorithm ~\ref{alg:1} is obtained by FISTA \cite{Beck_Teb_2013}, after suitable reformulation of the minimization problem.\\
Let us assume that the values $\lambda_i^{(k)}$, $i=1,\ldots,N$, and $\alpha^{(k)}$ are fixed, then  problem \eqref{eq:uno} can be written as:
\begin{equation}
 \min_{\mathbf{f}} \left \{\Psi_1(\mathbf{f})+ \Psi_2(\mathbf{f})\right \}
 \label{eq:due}
 \end{equation}
where:
\begin{equation*}
  \Psi_1(\mathbf{f})=  \left \|
\begin{pmatrix}
 \mathbf{K} \\
 \sqrt{\mathbf{\Lambda}^{(k)}} \mathbf{L}
 \end{pmatrix} \mathbf{f} - \begin{pmatrix}
 \mathbf{s} \\
 \mathbf{0}
 \end{pmatrix}
\right \|^2, \quad  \mathbf{\Lambda}^{(k)}=\text{diag}(\lambda_i^{(k)})
\end{equation*}
and
\begin{equation*}
  \Psi_2(\mathbf{f})= \alpha^{(k)} \| \mathbf{f} \|_1 .
\end{equation*}
The  FISTA steps for the solution of \eqref{eq:due} are reported in Algorithm \ref{alg:FISTA} where  $\xi$ is a constant stepsize and the starting guess corresponds to solution computed in  Algorithm \ref{alg:1} at the $k$-th step.
%
%-----------------------------------------------------------------------------------------------------
%
\begin{algorithm}[h]
\caption{-- $\mathbf{f}^{k+1}$ = \textsc{fista\_step}($\xi$,$\mathbf{f}^{(k)}$, $\Psi_1,\Psi_2$)) \label{alg:FISTA}}
\vskip 1mm
\begin{spacing}{1.15}
%{\small
\begin{algorithmic}[1]
\STATE Set $t_0=1$; $j=0$; $\mathbf{y}^{(1)}=\mathbf{f}^{(k)}$
\REPEAT
  \STATE $j=j+1$
\smallskip
  \STATE $\mathbf{f}^{(j)} = \arg\min_{\mathbf{f}} \left \{  \Psi_2(\mathbf{f}) + \frac{\xi}{2} \left \|  \mathbf{f} - \left (  \mathbf{y}^{(j)} - \frac{1}{\xi} \nabla(\Psi_1(\mathbf{y}^{(j)}))  \right )\right \|_2 \right \} $
\smallskip
  \STATE $t_{j+1} = \frac{1}{2}\left( 1+\sqrt{1+4t_j^2} \right)$
\smallskip
  \STATE $\mathbf{y}^{(j+1)}= \mathbf{f}^{(j)} + \displaystyle{\frac{(t_j^2-1)}{t_{j+1}}} \left(\mathbf{f}^{(j)}-\mathbf{f}^{(j-1)} \right)$
\smallskip
\UNTIL{$(\Psi_1(\mathbf{f}^{(j)}) + \Psi_2(\mathbf{f}^{(j)}) ) \leq \tau_\textsc{fista}$}
\smallskip
 \STATE  $\mathbf{f}^{(k+1)}=\mathbf{f}^{(j+1)} $
\end{algorithmic}
%} % end small
\end{spacing}
\end{algorithm}
%
% -----------------------------------------------------------------------------------------------------
%
At step 4 of Algorithm \ref{alg:FISTA}, the components of  $\mathbf{f}^{(j)}$  are computed explicitly, element-wise, by means of the soft thresholding operator:
\begin{equation*}
  \mathbf{f}^{(j)}_i = \text{sign}\left (z_i^{(j)}-\frac{\alpha}{\xi} \right) \max\left ( \left |z_i^{(j)} \right |-\frac{\alpha}{\xi} ,0\right ), \ \ i=1, \ldots, N
\end{equation*}
where
\begin{equation*}
  \mathbf{z}^{(j)}= \mathbf{y}^{(j)} -  \frac{1}{\xi} \nabla(\Psi_1(\mathbf{y}^{(j)})).
\end{equation*}
The convergence of FISTA has been proven for any stepsize $\xi$ such that $\xi\geq \mathcal{L}(\Psi_1)$, where $\mathcal{L}(\Psi_1)$ is the Lipschitz constant for the gradient $\nabla\Psi_1$  \cite{Beck_Teb_2013}; i.e:
\begin{equation} \label{eq:Leq}
   \mathcal{L}(\Psi_1) = \lambda_{\max} ( \mathbf{K}^T  \mathbf{K} + \mathbf{L}^T \mathbf{\Lambda}^{(k)} \mathbf{L})
\end{equation}
where $\lambda_{\max}(\mathbf{X})$ represents the maximum eigenvalue of the matrix $\mathbf{X}$. \\
The following theorem shows that an upper bound for $\mathcal{L}(\Psi_1)$ can be easily provided, thus obtaining the convergence of FISTA.
\begin{theorem}
Let  $\sigma_1^{(1)}$ and $ \sigma_1^{(2)} $ be the maximum singular values of the matrices $\mathbf{K}_1$ and $\mathbf{K}_2$, respectively, and let $\lambda_i^{(k)}$ be the local regularization parameters computed at $k$th step of Algorithm \ref{alg:1}, then
the value $\xi$ defined as follows:
\begin{equation}
   \xi=\left (\sigma_1^{(1)} \sigma_1^{(2)} \right )^2 + 64  \max_i | \lambda_i^{(k)} |
   \label{eq:LL}
\end{equation}
satisfies
\begin{equation*}
   \xi \geq  \mathcal{L}(\Psi_1)
\end{equation*}
and it guarantees the convergence of the FISTA method.
 \end{theorem}
%------------------------------------------------------------
 \begin{proof}
 By using equation \eqref{eq:Leq} we can majorize $\mathcal{L}(\Psi_1)$ as follows:
\begin{equation*}
  \lambda_{\max} ( \mathbf{K}^T  \mathbf{K} + \mathbf{L}^T \mathbf{\Lambda}^{(k)} \mathbf{L}) \leq  \lambda_{\max} ( \mathbf{K}^T  \mathbf{K}) + \lambda_{max} ( \mathbf{L}^T \mathbf{\Lambda}^{(k)} \mathbf{L})
\end{equation*}
%Let $\sigma_1^{(1)}$ and $\sigma_1^{(2)}$ be the maximum singular values of the matrices $\mathbf{K}_1$ and $\mathbf{K}_2$ respectively, then 
Using the Kronecker product properties of the Singular Value Decomposition (SVD), we have:
\begin{equation}
  \lambda_{\max}(\mathbf{K}^T \mathbf{K})= \lambda_{\max}( (\mathbf{K}_2 \otimes \mathbf{K}_1^T)^T (\mathbf{K}_2 \otimes \mathbf{K}_1^T)) = (\sigma_1^{(1)} \sigma_1^{(2)})^2.
\label{eq:K}
\end{equation}
Concerning the term  $\lambda_{\max}(\mathbf{L}^T \mathbf{\Lambda}^{(k)} \mathbf{L})$ we can apply the
property of the discrete Laplacian matrix $\lambda_{\max}(\mathbf{L}) \le 8 $  hence:
\begin{equation}
\lambda_{\max} ( \mathbf{L}^T \mathbf{\Lambda} \mathbf{L}) \le 64  \max_i | \lambda_i^{(k)} | .
\label{eq:L}
\end{equation}
Finally, collecting the terms \eqref{eq:K} and \eqref{eq:L}, we obtain the value $\xi$ that guarantees the convergence of FISTA steps.
\end{proof}
We observe that, in NMR, the matrices $\mathbf{K}_1$ and $\mathbf{K}_2$ are usually small size and their SVD can be easily performed in order to compute the values $\sigma_1^{(1)}$ and $\sigma_1^{(2)}$.

Following the observations in \cite{Bortolotti_2016}, we apply Algorithm \ref{alg:1}
with the following $L2$ penalty parameters, which have proven to be very efficient in NMR problems:
\begin{equation}\label{eq:lambda2}
    \lambda_i^{(k)} = \frac{\|\mathbf{K} \mathbf{f}^{(k)} - \mathbf{s} \|^2}{(N+1)\left ( \beta_0 +\beta_p
     \underset{\substack{\mu \in I_i}}\max \, (\mathbf{p}^{(k)}_{\mu})^2
 + \beta_c \underset{\substack{\mu \in I_i }} \max \, (\mathbf{c}^{(k)}_{\mu})^2\right )},  \quad i=1,\ldots,N
%    \lambda_i = \frac{\|\mathbf{K} \mathbf{f} - \mathbf{s} \|^2}{N\left ( \beta_0 +\beta_p
%     \Phi\big( (P_{i,j})^2\big)
% + \beta_c \Phi\big( (C_{i,j})^2\big)\right )},  \quad i=1,\ldots,N
\end{equation}
%\end{linenomath}
%
where
\begin{equation*}
  \mathbf{c}^{(k)}=\mathbf{L}\mathbf{f^{(k)}}, \quad
  \mathbf{p}^{(k)} = \text{vec} \big( \| \nabla \mathbf{F}^{(k)}\|\big), \quad \mathbf{f}^{(k)}=\text{vec}(\mathbf{F}^{(k)})
\end{equation*}
and $\mathbf{F}^{(k)}$ is the $k$-th distribution map (here, $\text{vec}(\mathbf{V})$ denotes the vector obtained by columnwise reordering the elements of a matrix $\mathbf{V}$).
The $I_i$ are the indices subsets related to the neighborhood of the point $i$
and the $\beta$'s are positive parameters; $\beta_0$ prevents division by zero and is a compliance floor,
which should be small enough to prevent under-smoothing, and large enough to avoid over-smoothing.
The optimum value of $\beta_0$, $\beta_c$ and $\beta_p$ can change with the nature of the measured sample.

Finally, the proposed procedure is stated in Algorithm \ref{alg:L1LL2} and is called L1LL2 method, which comes from "method based on $L1$ and Locally adapted $L2$ penalties''. As already discussed in \cite{Bortolotti_2016}, the starting guess $\f^{(0)}$ is computed by applying a few iterations of the Gradient Projection (GP) method to the nonnegatively constrained least squares problem
$$ \min_{\f \geq 0}  \|\mathbf{K} \mathbf{f} - \mathbf{s} \|^2 . $$

\begin{algorithm}[h]
\caption{ -- L1LL2 method \label{alg:L1LL2}}
\vskip 1mm
\begin{spacing}{1.15}
%{\small
\begin{algorithmic}[1]
\STATE Choose $\tau\in(0,1)$ and $\beta_0,\beta_p,\beta_c>0$
\STATE Set $k=0$ and compute $\mathbf{f}^{(0)}$
\STATE Compute $\sigma_1^{(1)}$ and $\sigma_1^{(2)}$ %SVDs of matrices $\mathbf{K}_1$ and $\mathbf{K}_2$;
\REPEAT
  \STATE Set $\epsilon^{(k)}=\|\mathbf{K} \mathbf{f}^{(k)} - \mathbf{s} \|^2$
\smallskip
  \STATE Set $\lambda_i^{(k)}=\displaystyle{\frac{\epsilon^{(k)}}{(N+1)\left ( \beta_0 +\beta_p
     \underset{\substack{\mu \in I_i}}\max \, (\mathbf{p}^{(k)}_{\mu})^2
 + \beta_c \underset{\substack{\mu \in I_i }} \max \, (\mathbf{c}^{(k)}_{\mu})^2\right )}}$, $i=1,\ldots,N$
\smallskip
  \STATE Set $\alpha^{(k)} = \displaystyle{\frac{\epsilon^{(k)}}{(N+1) \| \f ^{(k)}\|_1}}$
\smallskip
  \STATE Set $\xi^{(k)}=\left (\sigma_1^{(1)} \sigma_1^{(2)} \right )^2 + 64  \max_i | \lambda_i^{(k)} |$
\smallskip
  \STATE Compute
  $$\f ^{(k+1)}= \hbox{\textsc{fista\_step}}(\xi^{(k)},\mathbf{f}^{(k)},  \|\mathbf{K} \mathbf{f} - \mathbf{s} \|^2  + \sum_{i=1}^N \lambda^{(k)}_i(\mathbf{L}\mathbf{f})_i^2,\alpha^{(k)} \|\mathbf{f}\|_1) $$
\UNTIL{$\|\f ^{(k+1)} - \f ^{(k)} \| \leq \tau \|  \f ^{(k)} \|$}
\end{algorithmic}
%} % end small
\end{spacing}
\end{algorithm}
\section{Numerical Results \label{num-res}}
The analysis of the proposed algorithm is carried out in this section both on synthetic and real NMR relaxation data. The numerical tests  are performed on a PC laptop  equipped with 2,9 GHz Intel Core i7 quad-core, 16 GB RAM. The algorithms are implemented in Matlab R2019b. The values 
$\tau=10^{-3}$ and $\tau_\textsc{fista}=10^{-7}$ have been fixed for the stopping tolerances of L1LL2 (Algorithm \ref{alg:L1LL2}) and FISTA (Algorithm \ref{alg:FISTA}) respectively. 
\subsection{Synthetic data}
Aim of this paragraph is to draw some conclusions about the accuracy and performance of the proposed L1LL2 algorithm.  To this purpose we test L1LL2 algorithm on  synthetic  data that emulates the results of measurement with a 2D IR-CPMG sequence (see \cite{Bortolotti_2016, Bortolotti_2017, micro_meso2018}), by discretizing the  following Fredholm integral equation:
\begin{equation}
  S(t_1,t_2)=\iint_0^\infty k_1(t_1,T_1)k_2(t_2,T_2)F(T_1,T_2) \ dT_1 \ dT_2 + e(t_1,t_2)
\label{model}
\end{equation}
where 
$T_1$, $T_2$ are the  longitudinal and transverse
relaxation times related to the  evolution parameters $t_1$, $t_2$, and
the  kernels $k_1,k_2$ have the following expression:
\begin{equation}
 k_1(t_1,T_1) = 1- 2\exp(-t_1/T_1),  \ \ \ k_2(t_2,T_2)=\exp(-t_2/T_2).
 \label{eq:IRCPMG}
\end{equation}
Two different relaxation maps $F(T_1,T_2)$ are applied to obtain the synthetic relaxation  data. \\
The first relaxation map, named {\tt 2Pks} test, has size $N_1\times N_2$ where $N_1=N_2 =80$.  The  relaxation map, represented in Figure \ref{fig:SY_2}, has two peaks at positions $( T_1=814.97 \ ms,   T_2=4.533 \ ms)$ and $(T_1 = 119.54 \ ms, T_2=    8.5606 \ ms)$, relative to two population spins with $T_1 > T_2$.\\
The second relaxation map, named {\tt 3Pks} test, has  $100 \times 100$ points.
In this case the relaxation map, represented in  Figure \ref{fig:SY_3}, has three peaks, relative to  three population spins, with peaks in the following positions: $(T_1=1582.2 \ ms,    T_2=32.289 \ ms)$, $(T_1=5.9692 \ ms,   T_2=2.6124 \ ms)$ and $(T_1=1139.5 \ ms,   T_2= 258.08 \ ms)$. \\
In both cases the length of the IR sequence is $M_1= 128$ while the CPMG sequence has length $M_2 = 2048$.\\
Normal Gaussian random noise  $\mathbf{e}$ of level $\delta\equiv\|\mathbf{e}\|$ is added as follows:
%\begin{linenomath}
$\mathbf{s}  = \mathbf{K} \mathbf{f^*}+ \mathbf{e}$.
%
%--------------------------------------------------------------------------------------------------
% Figure 1
\begin{figure}
\centering
\subfigure[Two peaks map]{\label{fig:SY_2}\includegraphics[width=0.48\textwidth]{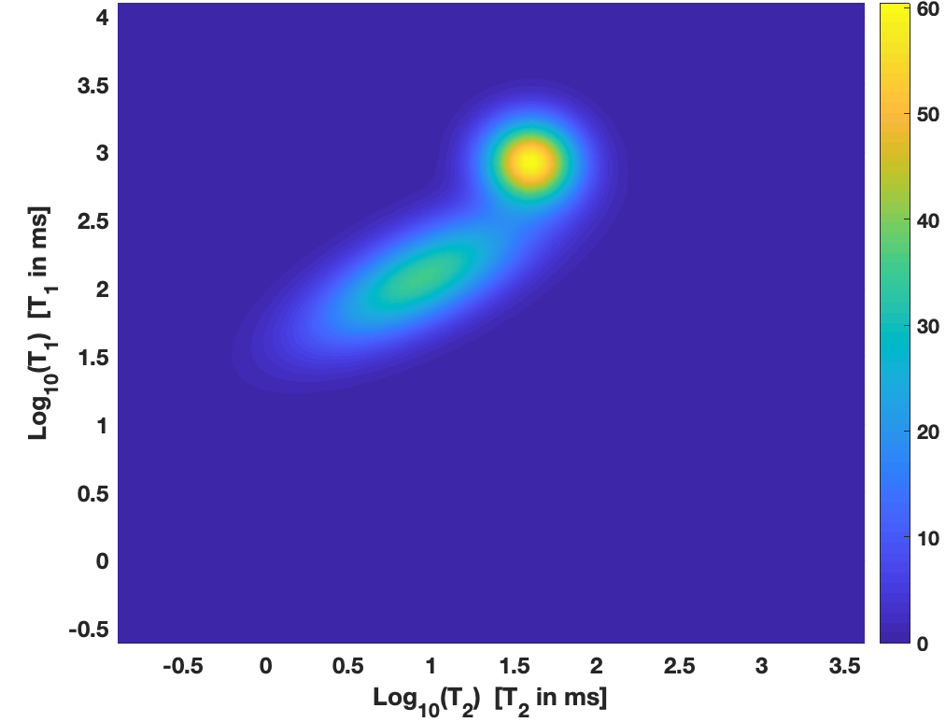}}
%map]{\label{fig:SY_2}\includegraphics[scale=1.0]{Pk2_n.png}}
\subfigure[Three peaks map]{\label{fig:SY_3}\includegraphics[width=0.48\textwidth]{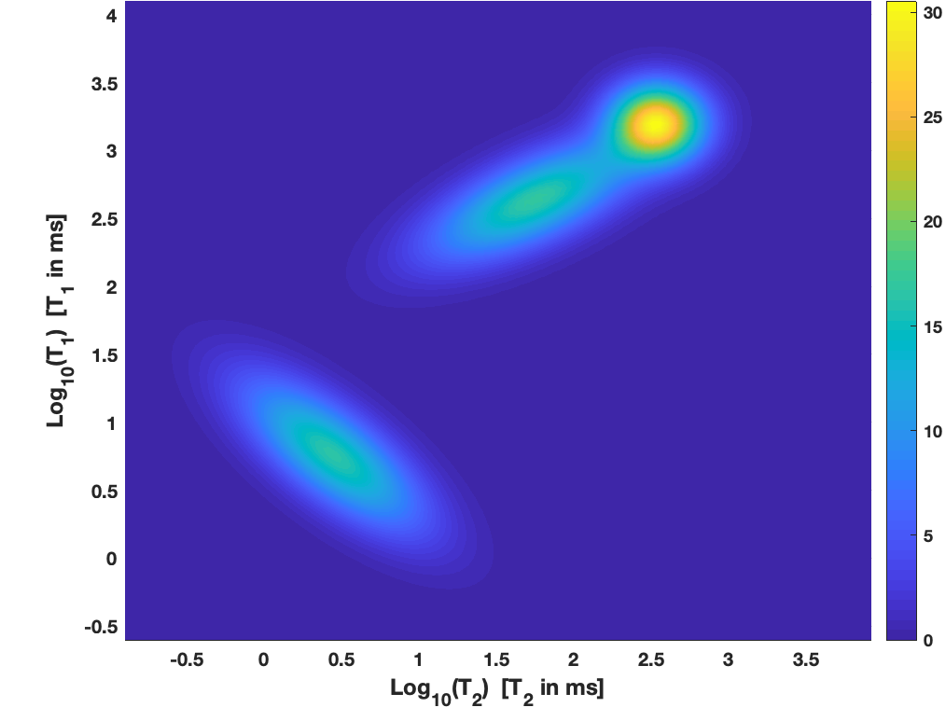}}
%map]{\label{fig:SY_3}\includegraphics[scale=1.0]{Pk3_n.png}}
\caption{Maps of relaxation times used in tests with synthetic data.}
\label{fig:Synth_maps}
\end{figure}
%
%----------------------------------------------------------------------
%
The tests are carried out by running $10$ noise realizations
with  $\delta =10^{-2} $ and the reported numerical results are  averaged over these noise realizations.\\
The algorithms accuracy is measured by means of the relative error $Erel^2$ and  the Root Mean Squared Deviation $RMSD$, defined as follows:
\begin{equation}
Erel^2 = \| \mathbf{f} - \mathbf{f}^* \|^2 /  \| \mathbf{f}^* \|^2, \ \ \
RMSD = \frac{\| \mathbf{\hat s}  - \mathbf{s} \|}{\sqrt{M}}, \ \ \ \mathbf{\hat s} =\mathbf{K} \mathbf{ f}
\end{equation}
where $\mathbf{f}$ represents the map computed by the algorithms and $\mathbf{f}^*$ is the true map.
The first analysis evaluates the effects of the multi-penalty and multi-parameter approach (L1LL2) compared to the L1 and multi-parameter L2 penalties.
The  adapted  L1 penalty algorithm (A\_L1) is obtained applying  algorithm \ref{alg:1}  to solve \eqref{eq:uno}  with $\lambda_i =0 $, 
while multi-parameter L2  is obtained by solving problem \eqref{eq:L2} with algorithm 2DUPEN.
In table \ref{tab:synth}, the error parameters and computation times are reported for each algorithm. 
While 2DUPEN reaches always the most accurate solutions, the A\_L1 algorithm has the greatest relative error values. Conversely, A\_L1 is the fastest method while 2DUPEN requires the longest computation times.
%--------------------------------------------------------------------
%  Table 1
%
\begin{table}
\centering
\begin{tabular}{clccccc}
\multirow{2}{*}{Test} & \multirow{2}{*}{Algorithm}        & $Erel^2 $ & $RMSD$  & Time \\
                                &                                                  &  (-)              & (a.u.)  & (s) \\
\hline
     & L1LL2 & $1.22 \ 10^{-1} $ & $1.953 \ 10^{-4}$  &  $10.80$ \\
2Pks & A\_L1 & $1.41 \ 10^{-1} $& $ 1.953 \ 10^{-4}$  &  $10.13 $\\
     & 2DUPEN       & $8.79 \ 10^{-2}$ & $1.953 \ 10^{-4}$ & $386 $\\
%   & GP & $1.58 \ 10^{-1} $ & $1.58 \ 10^{-2} $  &  $440.7 \ s$\\
%& opt-L1$^{(+)}$ & 1.2741e-1& 1.5617e-2 & \\
\hline
%3Pk
           & L1LL2 & $1.09 \ 10^{-1} $ & $1.381 \ 10^{-4}$  & $29.00 $\\
 3Pks      & A\_L1  &   $1.31 \ 10^{-1} $ & $1.381 \ 10^{-4}$ & $ 19.84  $\\
           & 2DUPEN      & $8.51 \ 10^{-2}$& $1.381 \ 10^{-4}$ &  $85.62 $\\
%   & GP & $1.82 \ 10^{-1} $ & $1.44 \ 10^{-2} $  &  $ 1797.9 \ s$\\
 %      & opt-L1 $^{(+)}$& 1.1541e-1 & 9.999e-3 & \\
\hline
\end{tabular}
\caption{Accuracy and computation times of  the  synthetic tests. Reference value $RMSD^*=\delta / \sqrt{M}  =  1.9531 \ 10^{-5}$. }
\label{tab:synth}
\end{table}
%-----------------------------------------------------------
%
 We observe that L1LL2 achieves the best trade-off between accuracy and computation time.
We can quantify such trade-off in terms of Percentage Accuracy Loss (PAL), obtained subtracting the $Erel^2$ of 
2DUPEN,  to that of each algorithm:
\begin{equation}
PAL_m = 100 \frac{Err_m- Err_{min}}{Err_{min}}
\label{eq:PAL}
\end{equation}
where  $Err_m$ represents  the relative error of method $m$ (L1LL2 or L1) and $Err_{min}$ is the minimum relative error, always obtained by 2DUPEN. 
Analogously, we measure the Percentage Efficiency Gain (PEG) by subtracting the computation time of 
 L1LL2 or L1  ($Time_m$) to that 2DUPEN ($Time_{max}$):
 \begin{equation}
PEG_m = 100 \frac{Time_{max}-Time_m}{Time_{max}}
\label{eq:PEG}
\end{equation}
The values reported in table \ref{tab:effi} show that for the 2pks test the accuracy lost by L1LL2 is about $22\%$ smaller than L1 while the computation efficiency  gained is similar. Concerning the 3Pks test we observe that 
accuracy lost by L1LL2 is about $25\%$ smaller than L1 while the performance gained by L1  is $ 11\%$  greater than that of L1LL2. Hence   L1LL2 reaches the best balance between accuracy and computation time. 
%-----------------------------------------------------------------------
\begin{table}
\centering
\begin{tabular}{cccc}
Test & Algorithm        & PAL  & PEG  \\
\hline
 \multirow{2}{*}{2Pks}   & L1LL2 & $38.8 \%$  & $97.2 \%  $ \\
           &       A\_L1 & $60.5 \% $&  $97.4  \%$ \\
  \multirow{2}{*}{3Pks}   & L1LL2 & $28.1  \% $& $66.1 \% $ \\
           &       A\_ L1 & $53.9  \% $& $78.0 \%$ \\
\hline
\end{tabular}
\caption{Percentage Accuracy loss, PAL \eqref{eq:PAL}, and percentage performance gain (PEG \eqref{eq:PEG})  obtained by each method on the different test problems.}
\label{tab:effi}
\end{table}
%-----------------------------------------------------------------------
This feature is well represented in figures \ref{fig:errs}  where the time evolution of  $Erel^2 $ is plotted  for each algorithm.
We observe that the introduction of the  $L_1$ regularization causes a considerable decrease in  
the total computation time at the expenses of a slight increase in the relative error. \\
\begin{figure}[hbtp]
\centering
\subfigure[Two peaks test]{\label{fig:E_2}\includegraphics[width=0.48\textwidth]{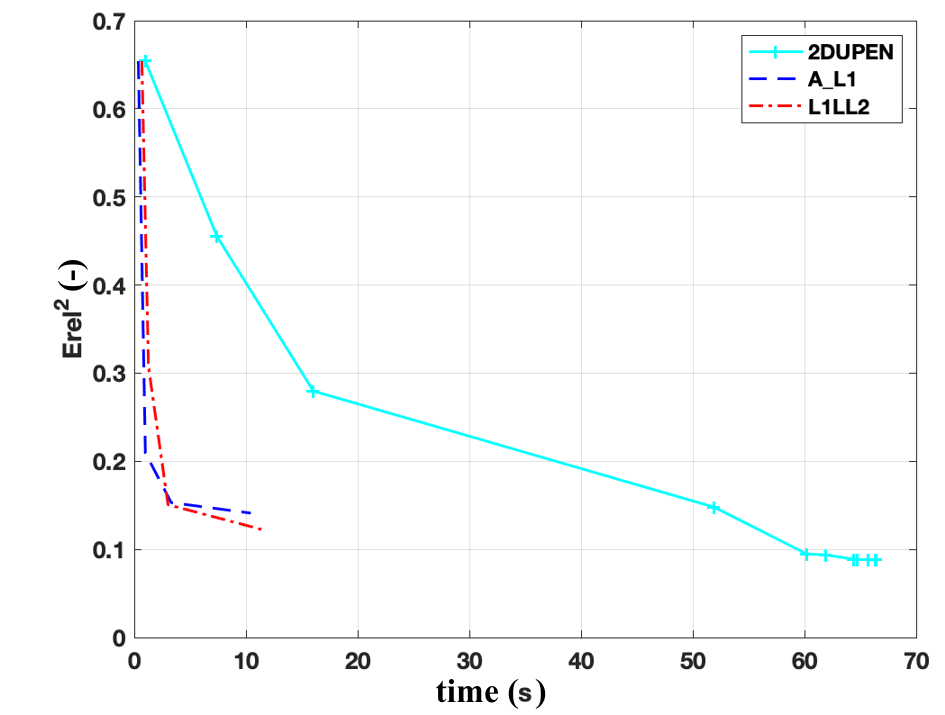}}
\subfigure[Three peaks test]{\label{fig:E_3}\includegraphics[width=0.48\textwidth]{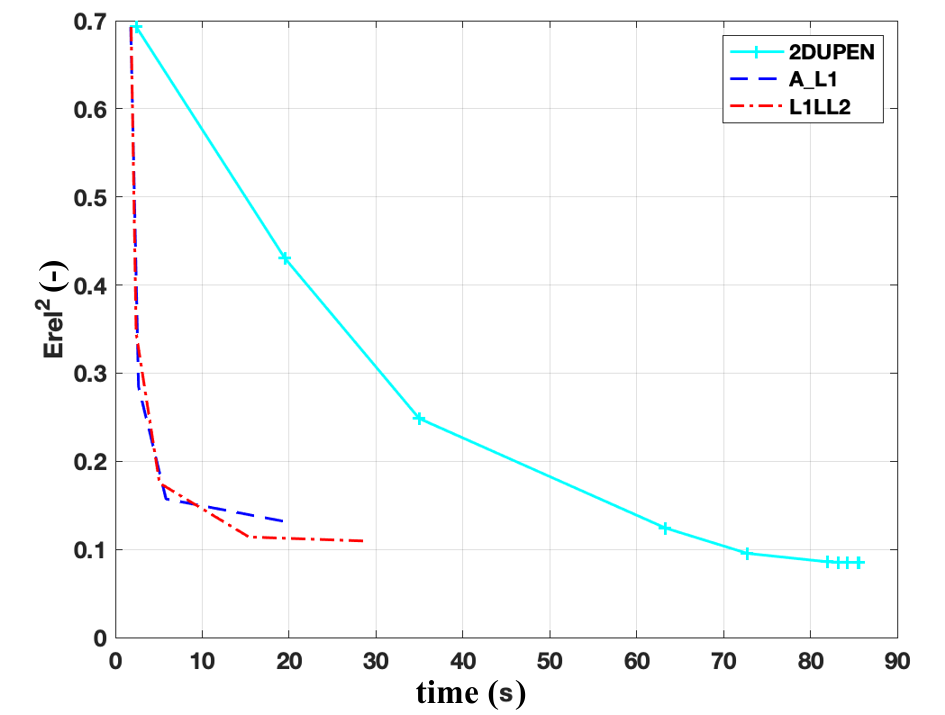}}
\caption{Relative Error vs. computation time: 2DUPEN cyan line, A\_L1, blue dashed line, L1LL2 red dash-dotted line.}
\label{fig:errs}
\end{figure}
The contribution of the multi-parameter $L2$ term to the algorithm accuracy  is further highlighted in figures \ref{fig:Cont1_2pks} and \ref{fig:Cont1_3pks}.
We observe that figures (b) and (d) are more precise in reproducing the true contour levels (figure (a)), compared to  the A\_L1 regularization algorithm in figure (c). \\
Concerning the residual values, we observe that the  RMSD parameters reported   in table \ref{tab:synth} have indeed very tiny differences, in the range    $[10^{-8} , 10^{-7}]$,  revealing  equal data consistency  for all methods.  Moreover the good results are confirmed  by  the RMSD correspondent to the true solution $\mathbf{f}^* $, given by  $RMSD^*=\delta / \sqrt{M}  =  1.9531 \ 10^{-5}$.  
%
%------------------------------------------------------------------------------------
\begin{figure}[hbtp]
\centering
\subfigure[True]{\includegraphics[width=0.48\textwidth]{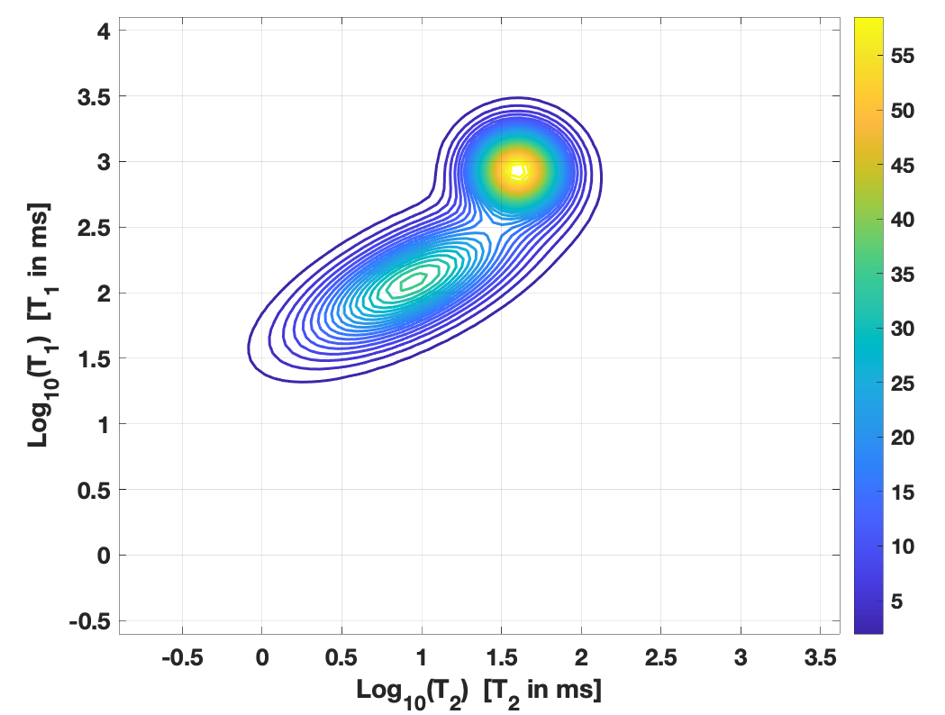}}
\subfigure[L1LL2]{\includegraphics[width=0.48\textwidth]{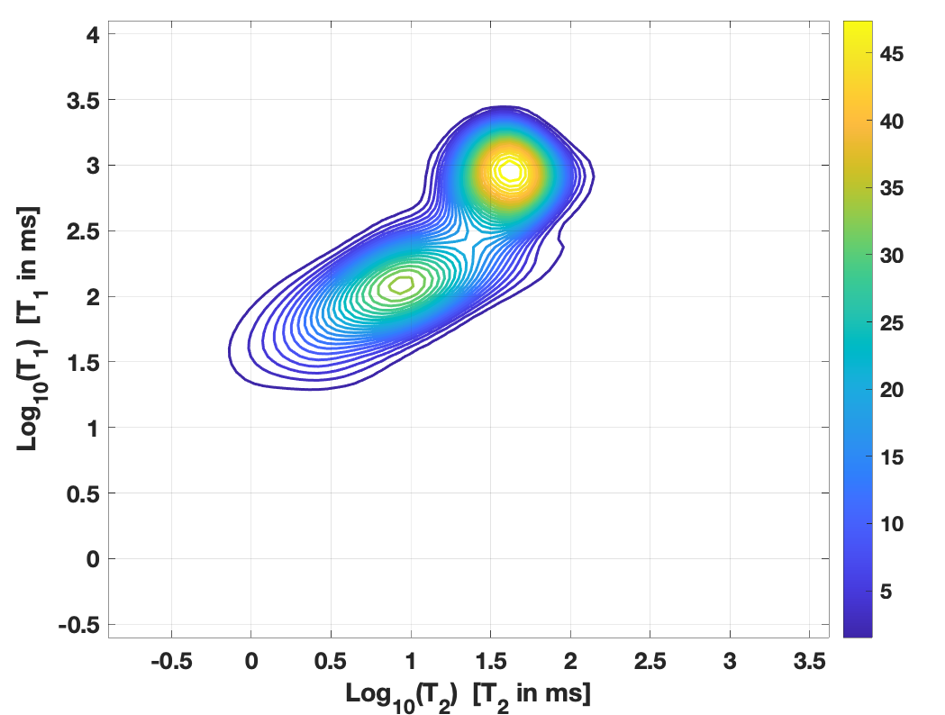}} \\
\subfigure[A\_L1]{\includegraphics[width=0.48\textwidth]{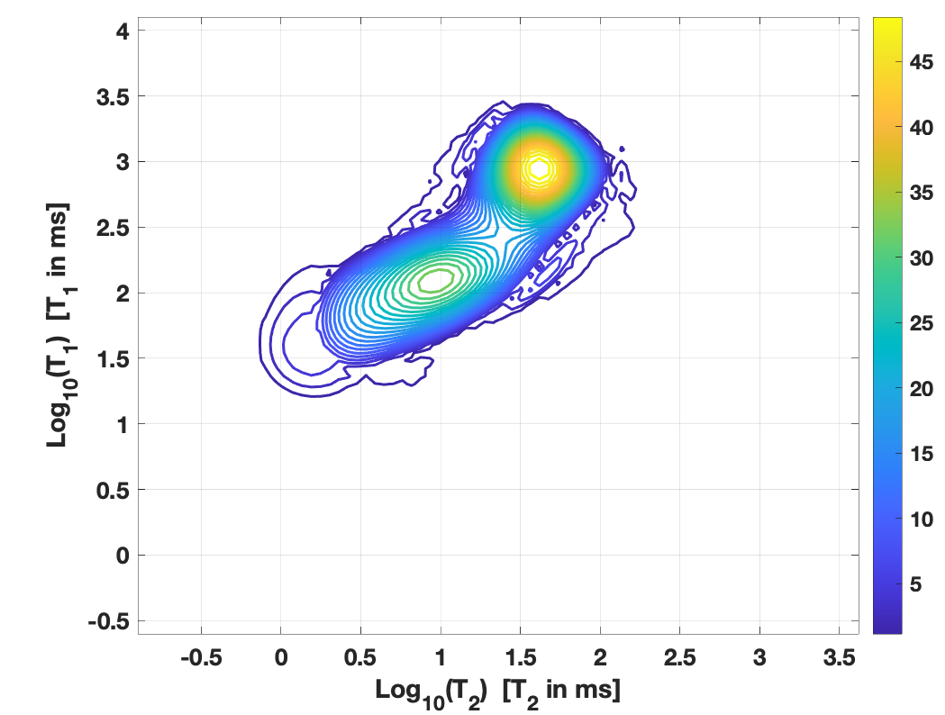}}
\subfigure[2DUPEN]{\includegraphics[width=0.48\textwidth]{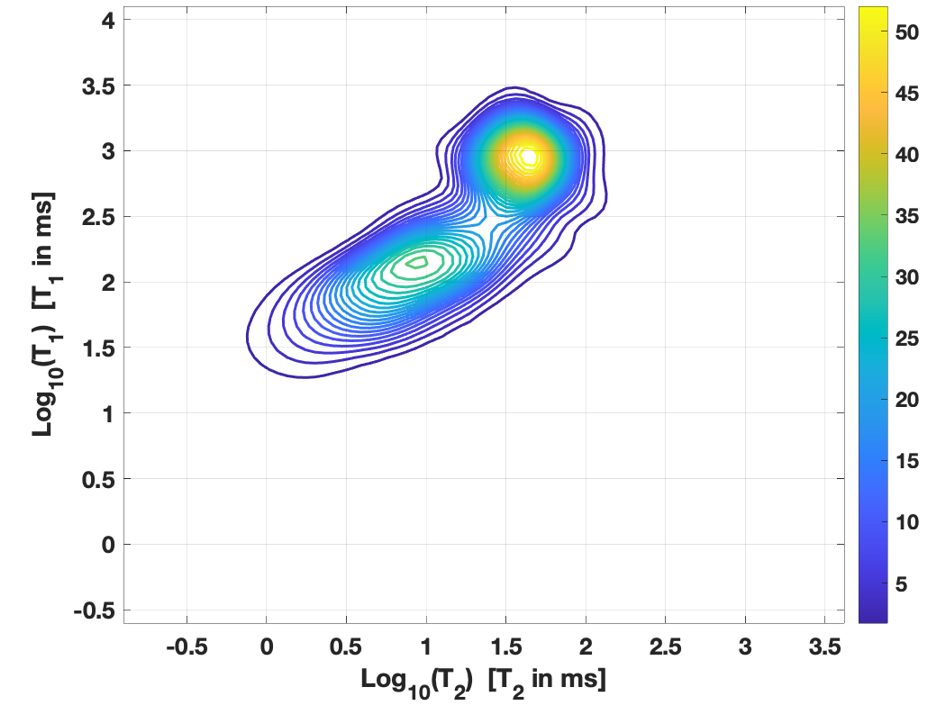}}
\caption{2pks test contour plots }\label{fig:Cont1_2pks}
\end{figure}
%------------------------------------------------------------------------------------
\begin{figure}[hbtp]
\centering
\subfigure[True]{\includegraphics[width=0.48\textwidth]{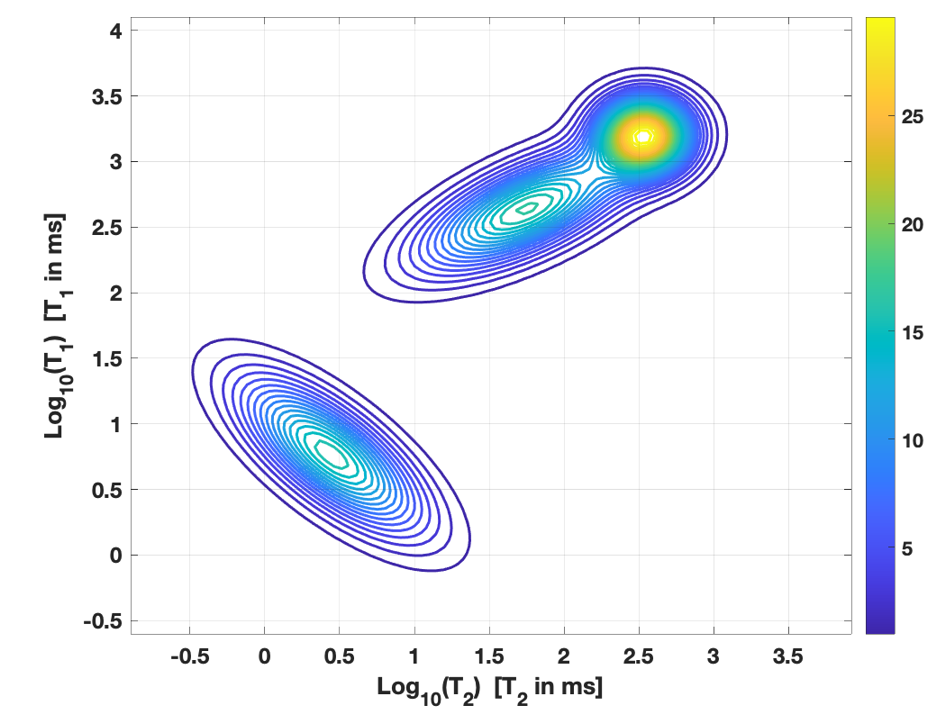}}
\subfigure[L1LL2]{\includegraphics[width=0.48\textwidth]{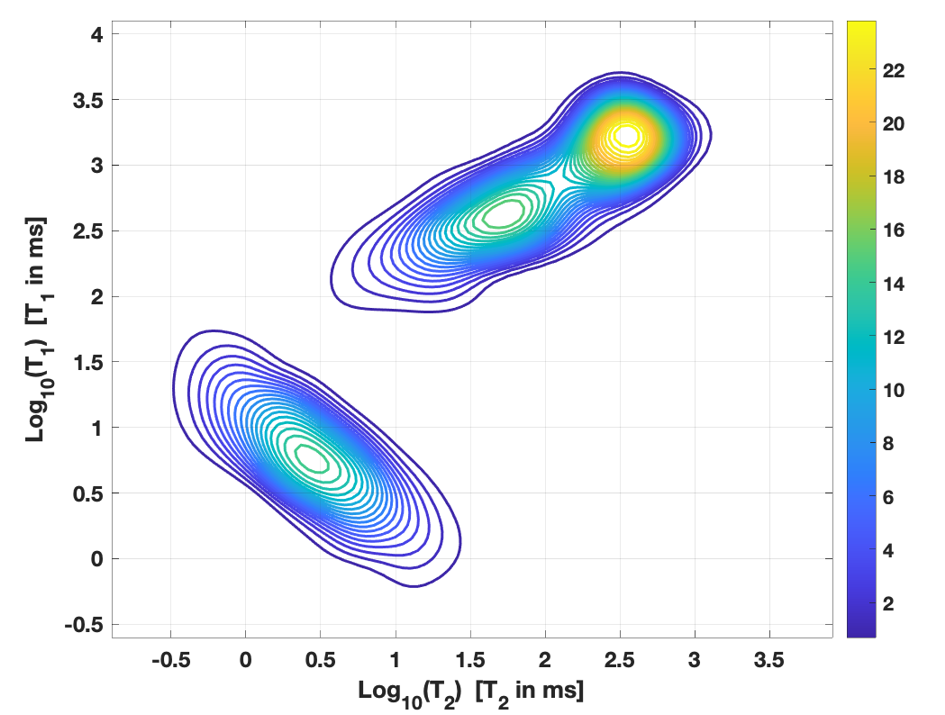}} \\
\subfigure[A\_L1]{\includegraphics[width=0.48\textwidth]{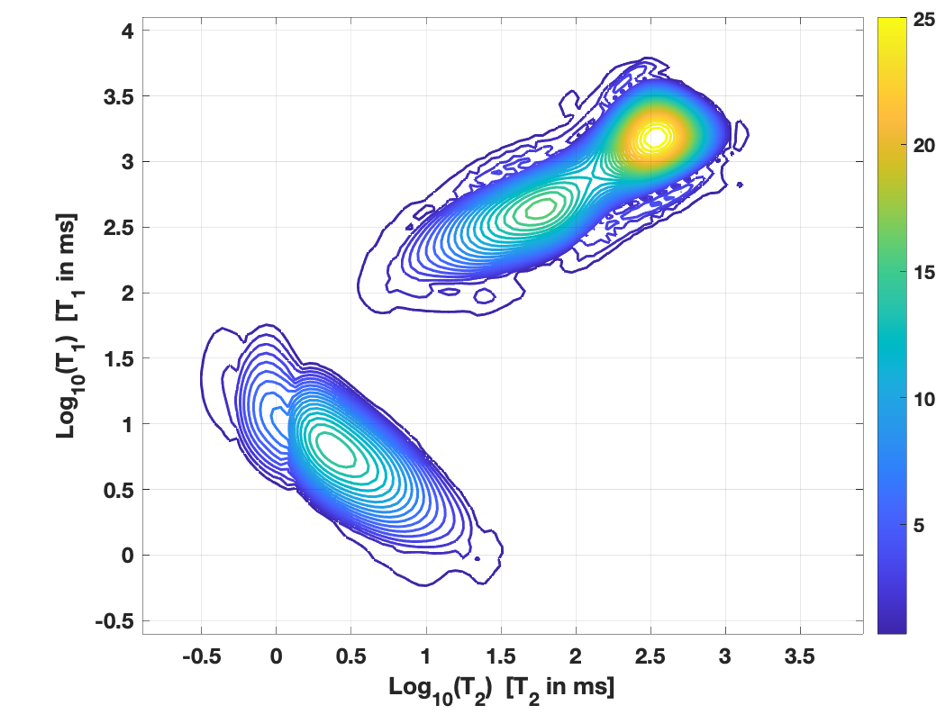}}
\subfigure[2DUPEN]{\includegraphics[width=0.48\textwidth]{Cont_Upen_L1_3pks.png}}
\caption{3pks test contour plots}\label{fig:Cont1_3pks}
\end{figure}
\subsection{Real data}
In this section we report the reconstructions obtained by L1LL2 with real data acquisitions (see  \cite{Bortolotti2019}, 
for a  detailed description of the samples  and acquisition modalities).\\
% Descrizione prob test
%
The first test, named $T_1-T_2$ test,  is relative to an  IR-CPMG sequence, with kernels 
defined in \eqref{eq:IRCPMG}, of $48 \times 1000$ data points, and  reconstructed relaxation map with $80 \times 80$ points.
The second test, named $T_2-T_2$ test is related to a CPMG-CPMG sequence, where the kernels in \eqref{model} have the following expressions: 
 $$k_1(t_1,T_{1})= \exp(-t_1/T_{1}),  \ \ \ k_2(t_2,T_{2})=\exp(-t_2/T_{2}) . $$
 %NoN si dice da nessuna parte chi sono $t_1,T_{1},t_2,T_2$!!\\
The reconstructed map has  $64 \times 64$ points  while the data sequence has  $128 \times 2800$ elements.
Since 2DUPEN has produced  the most accurate solutions on  different sequences and samples (see \cite{Bortolotti_2016}, \cite{Bortolotti_2017}, \cite{micro_meso2018} and \cite{Bortolotti2019})
we use it  as a reference method. \\
%-----------------------------------------------------------------------
\begin{table}
\centering
\begin{tabular}{cccc}
\multirow{2}{*}{Test} & \multirow{2}{*}{Algorithm}        & RMSD  & Time  \\
       &                         & (a.u.)  & s. \\
\hline
 \multirow{2}{*}{$T_1-T_2$}   & L1LL2 & $3.328 \ 10^{-3} $  & 3.5 \\
           &       2DUPEN & $3.333 \ 10^{-3} $ &   54.6 \\
  \multirow{2}{*}{$T_2-T_2$}   & L1LL2 & $1.248 \ 10^{-1} $ & 3.6  \\
           &       2DUPEN & $1.250 \ 10^{-1} $&  10.6  \\
\hline
\end{tabular}
\caption {RMSD and computation times obtained with 2D NMR real data.}
\label{tab:Tab_data}
\end{table}
%-----------------------------------------------------------------------
We observe in table \ref{tab:Tab_data} that for each test problem the RMSD values are very similar, confirming that L1LL2 preserves the data consistency as 2DUPEN. 
Moreover, the computation times show the improved efficiency of L1LL2, as expected.\\
In case of of $T_1-T_2$ test we can see the optimal correspondence with 2DUPEN in the 1D maps projections reported in figure \ref{fig:T1_A}
where both peaks are well localized in position, height and amplitude.
The contour maps of the computed 2D relaxation  time distributions, reported in figure \ref{fig:T1T2_data},
also confirm the good accuracy of L1LL2 compared to 2DUPEN.\\
Concerning the $T_2-T_2$ test, we again observe in table \ref{tab:Tab_data} the preservation of data consistency.    The contour maps of the relaxation times, reported in figure \ref{fig:T2T2_data}, 
report a precise reproduction of  spin population with the largest relaxation times, confirmed by the projection on the horizontal axis 
in figure \ref{fig:T1_B_2}. 
%-----------------------------------------------------------------------------
\begin{figure}[hbtp]
\centering
\subfigure[L1LL2]{\label{fig:t1_t2_1} \includegraphics[width=0.47\textwidth]{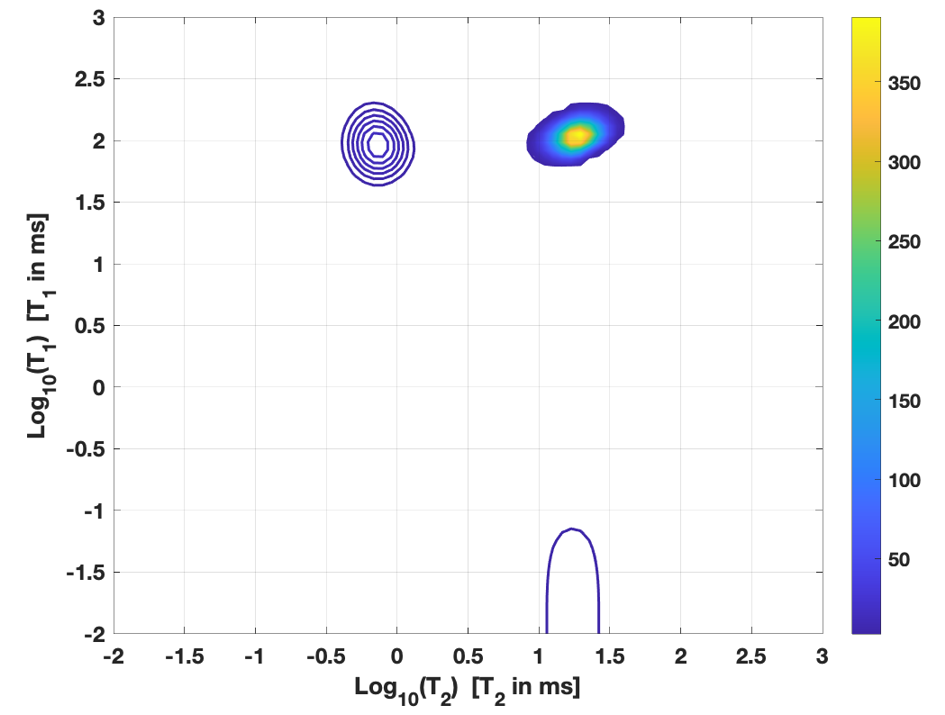} }
\subfigure[2DUPEN]{\label{fig:t1_t2_2} \includegraphics[width=0.47\textwidth]{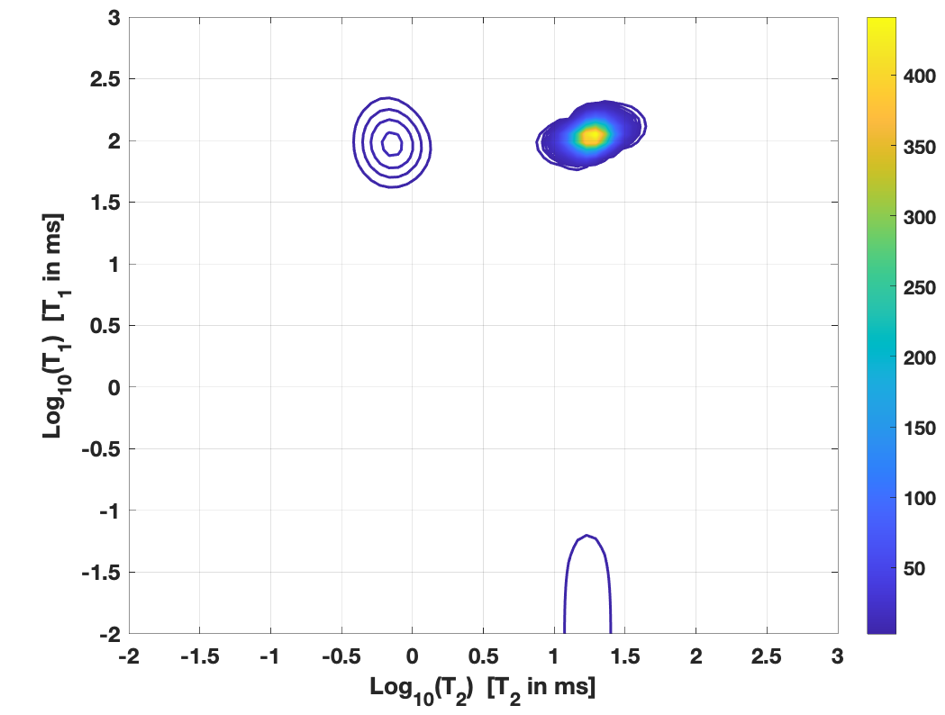} }
\caption{Test $T_1-T_2$, contour maps of relaxation times. (a) computed in $3.53 \ s$ (b) computed in
$54.6 \ s.$}
\label{fig:T1T2_data}
\end{figure}
\begin{figure}[hbtp]
\centering
\subfigure[$T_1$]{\label{fig:T1_A_1} \includegraphics[width=0.47\textwidth]{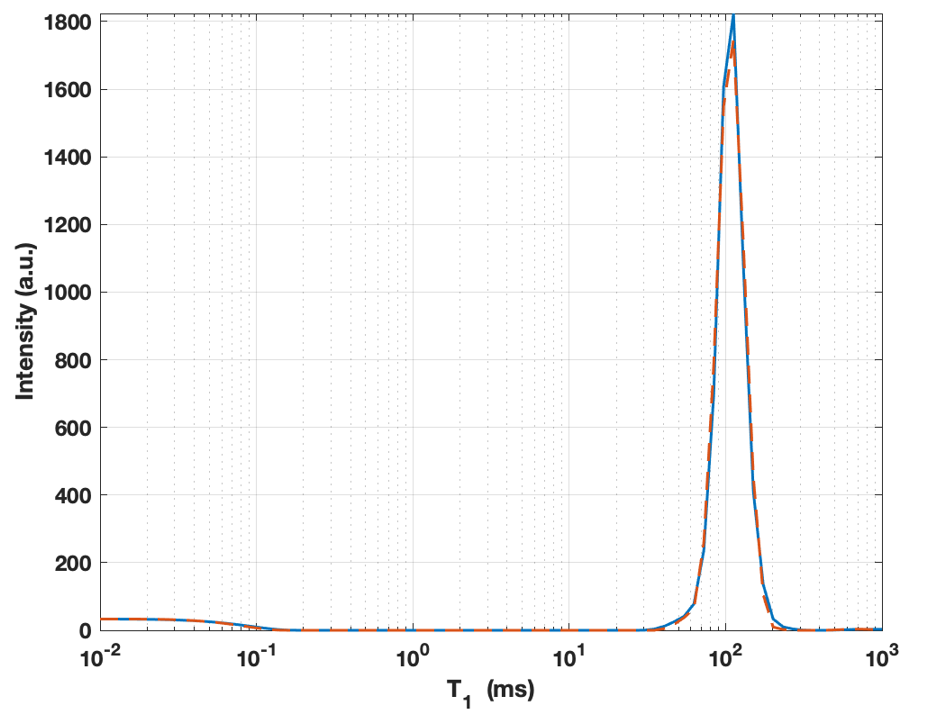}}
\subfigure[$T_2$]{\label{fig:T1_A_2} \includegraphics[width=0.47\textwidth]{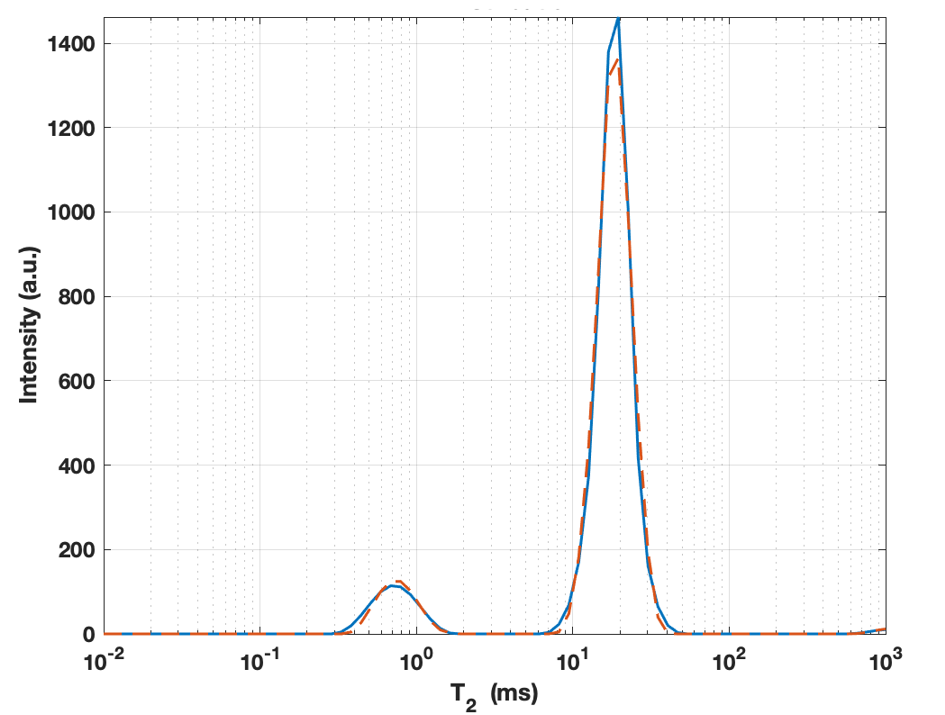}}
\caption{$T_1-T_2$ test, projection of the time relaxation maps onto the vertical axis $T_1$ and onto the horizontal axis $T_2$.L1LL2 red dashed line, 2DUPEN blue line.}
\label{fig:T1_A}
\end{figure}
Concerning two smaller populations it is evident in figure \ref{fig:T1_B_1}  that relaxation maps are not coincident, however in this case
L1LL2 provides a better separation of the spin populations with smaller relaxation times.
%The localization of the two smaller populations are  approximated with minor accuracy,
%as evidenced by  the projection onto the vertical axis, in figure \ref{fig:T1_B_1} where the smallest spins populations are not perfectly coincident.\\
Finally  we observe again that L1LL2  is computationally more efficient than 2DUPEN.

%-----------------------------------------------------------------------------
\begin{figure}[hbtp]
\centering
\subfigure[L1LL2]{\label{fig:t2_t2_1} \includegraphics[width=0.47\textwidth]{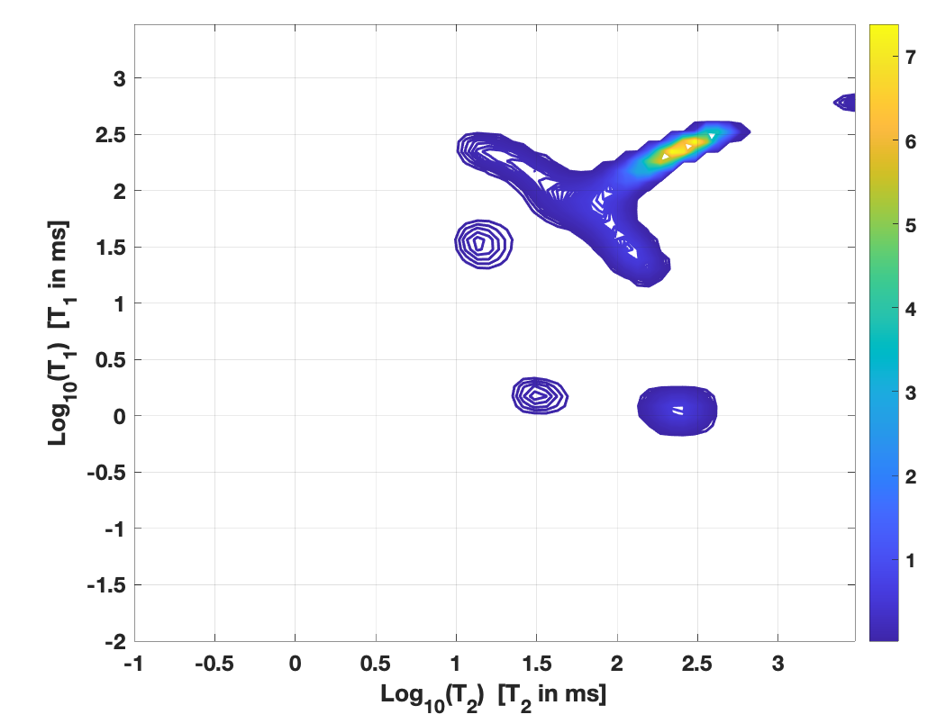} }
\subfigure[2DUPEN]{\label{fig:t2_t2_2} \includegraphics[width=0.47\textwidth]{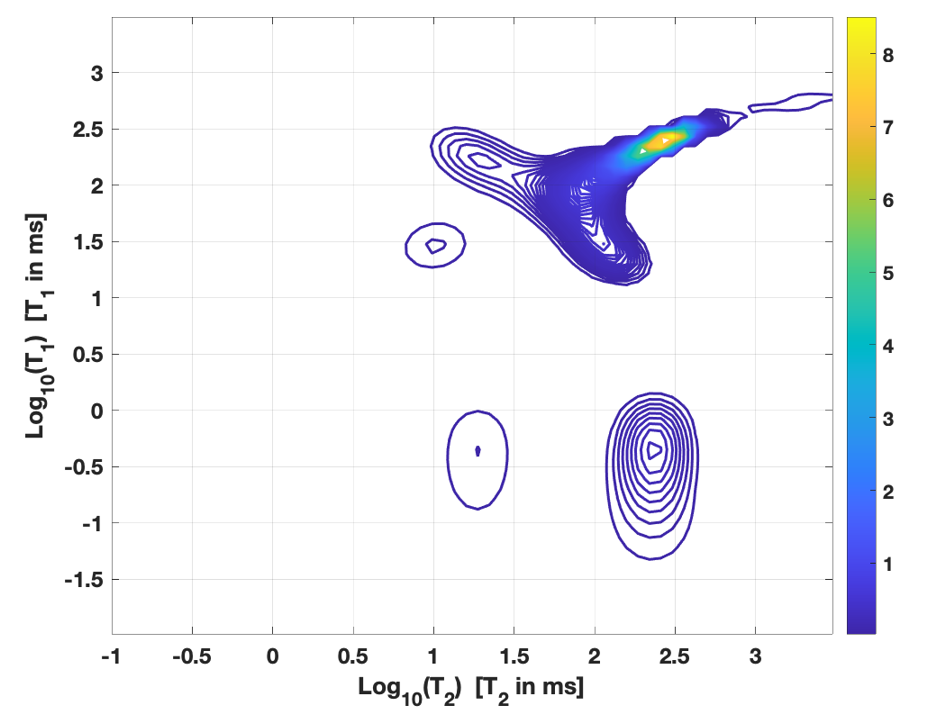} }
\caption{Test $T_2-T_2$, contour maps of relaxation times. (a) computed in $3.6 \ s$ (b) computed in $10.3 \ s$}
\label{fig:T2T2_data}
\end{figure}
%-----------------------------------------------------------------------------------------------
\begin{figure}[hbtp]
\centering
\subfigure[$T_{21}$]{\label{fig:T1_B_1} \includegraphics[width=0.48\textwidth]{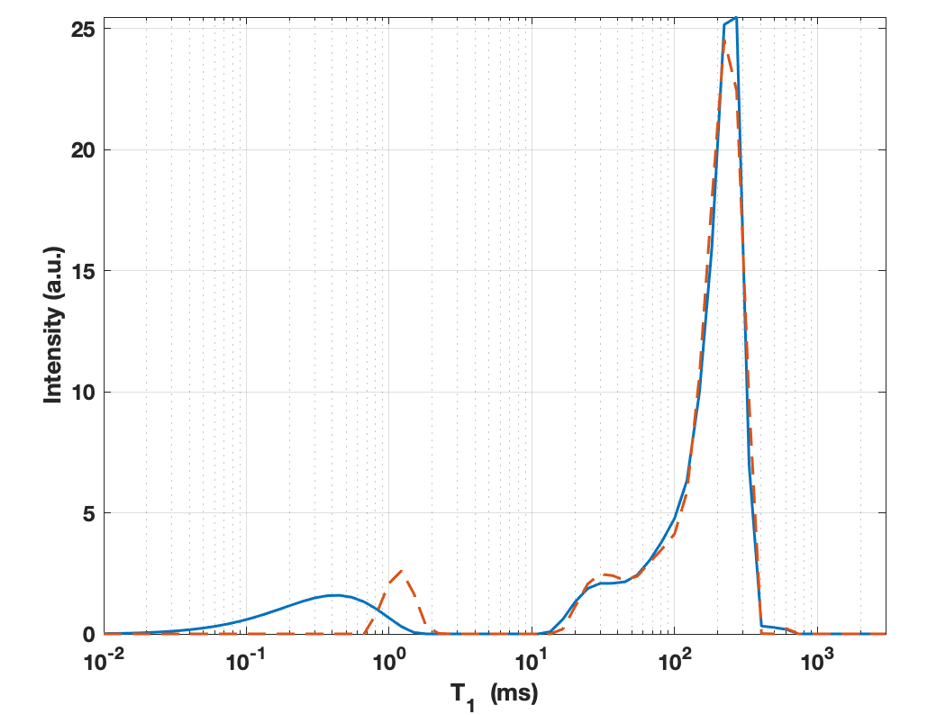}}
\subfigure[$T_{22}$]{\label{fig:T1_B_2} \includegraphics[width=0.48\textwidth]{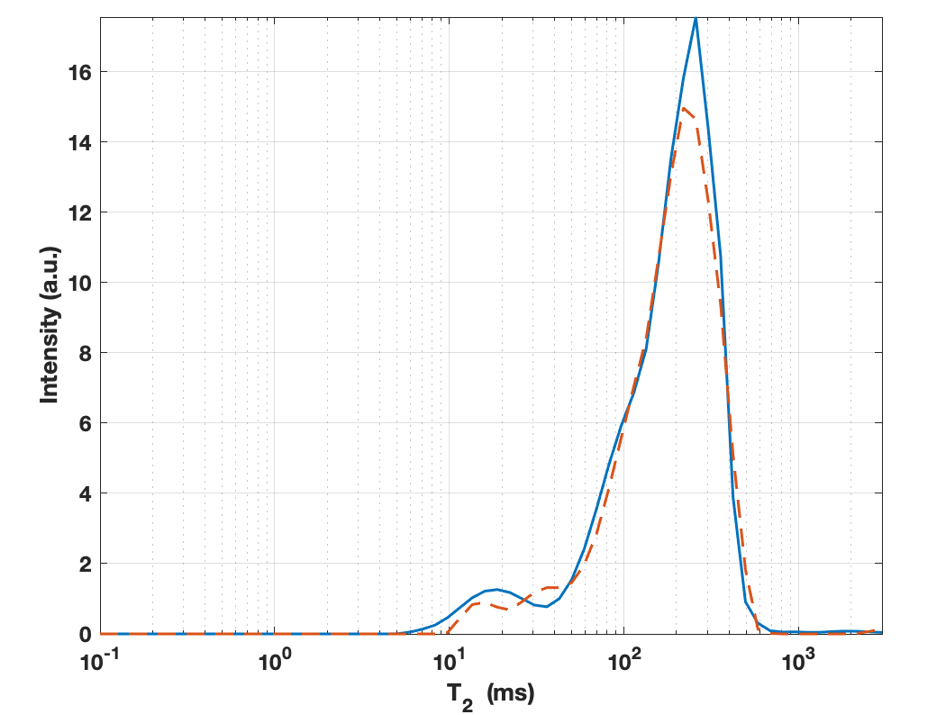}}
\caption{T2-T2 test, projection of the time relaxation maps onto the vertical axis $T_{21}$ and onto the horizontal axis $T_{22}$. L1LL2 red dashed line, 2DUPEN blue line.}
\label{fig:T1_B}
\end{figure}

\section{Conclusion}
This paper presents the  L1LL2  method for the inversion of 2DNMR relaxation data.
The algorithm automatically computes a 2D distribution of relaxation times and spatially adapted regularization parameters by iteratively solving a sequence of multi-penalty problems. The FISTA method is used for the solution of the minimization problems, and all the regularization parameters are updated according to
the uniform penalty principle. The L1LL2 method has been compared to A\_L1 and 2DUPEN. 
The numerical results show that, compared to 2DUPEN, the main advantage of L1LL2 is the increased computational speed without a significant loss in the inversion accuracy; compared to A\_L1, L1LL2 provides more accurate distributions at a comparable computational cost. 
Future work will consider the extension of this method to different penalty functions to employ the generalized uniform Penalty principle to a broader variety of inverse problems.

\section*{Acknowledgements}
This work was partially supported by Gruppo Nazionale per il Calcolo Scientifico - Istituto Nazionale di Alta Matematica (GNCS-INdAM).

%%%%\bibliography{biblio_arti.bib}

\end{document}